\documentclass[12pt,a4paper,reqno]{amsart}

\usepackage{amsfonts,amsthm,latexsym,amsmath,amssymb,amscd,epsfig}
\usepackage{graphics,graphicx}
\usepackage[all]{xy}
\usepackage{young}

\usepackage{tikz}
\usetikzlibrary{matrix}

\newtheorem{theorem}{Theorem}[section]
\newtheorem{lemma}[theorem]{Lemma}

\theoremstyle{definition}

\newtheorem{proposition}[theorem]{Proposition}
\newtheorem{corollary}[theorem]{Corollary}
\theoremstyle{remark}

\numberwithin{equation}{section}

\newcommand{\nc}{\newcommand}
\renewcommand{\frak}{\mathfrak}
\providecommand{\cal}{\mathcal}
\renewcommand{\bold}{\mathbf}

\nc \Ab{{\ensuremath{\bold A}}}
\nc \ab{{\ensuremath{\bold a}}}
\nc \bb{{\ensuremath{\bold b}}}
\nc \cb{{\ensuremath{\bold c}}}
\nc \Bb{{\ensuremath{\bold B}}}
\nc \Gb{{\ensuremath{\bold G}}}
\nc \Qb{{\ensuremath{\bold Q}}}
\nc \Rb{{\ensuremath{\bold R}}} \nc \Cb{{\ensuremath{\bold C}}} 
\nc \Eb{{\ensuremath{\bold E}}}
\nc \eb{{\ensuremath{\bold e}}}
\nc \Db{{\ensuremath{\bold D}}}
\nc \Fb{{\ensuremath{\bold F}}}
\nc \ib{{\ensuremath{\bold i}}}
\nc \jb{{\ensuremath{\bold j}}}
\nc \kb{{\ensuremath{\bold k}}}
\nc \nb{{\ensuremath{\bold n}}}
\nc \rb{{\ensuremath{\bold r}}}
\nc \Pb{{\ensuremath{\bold P}}}
\nc \pb{{\ensuremath{\bold p}}}
\nc \SPb{{\ensuremath{\bold {SP}}}}
\nc \Zb{{\ensuremath{\bold Z}}} 
\nc \zb{{\ensuremath{\bold z}}} 
\nc \gb{{\ensuremath{\bold g}}} 
\nc \fb{{\ensuremath{\bold f}}} 
\nc \ub{{\ensuremath{\bold u}}} 
\nc \vb{{\ensuremath{\bold v}}} 
\nc \yb{{\ensuremath{\bold y}}} 
\nc \xb{{\ensuremath{\bold x}}} 
\nc \xib{{\ensuremath{\bold \xi}}} 
\nc \Nb{{\ensuremath{\bold N}}} 
\nc \Hb{{\ensuremath{\bold H}}} 
\nc \wb{{\ensuremath{\bold w}}} 
\nc \Wb{{\ensuremath{\bold W}}} 
\nc \syz{{\mathbf {syz}}}
\nc \bnoll{{\ensuremath{\bold 0}}} 

\nc \mf{\frak m} \nc \mh{\hat{\m}} 
\nc \nf{\frak n}
\nc \Of{\frak O}
\nc \rf{\frak r}
\nc \mufr{{\mathbf \mu}}
\nc \hf{\frak h} 
\nc \qf{\frak q} 
\nc \bfr{\frak b} 
\nc \kfr{\frak k} 
\nc \pfr{\frak p} 
\nc \af{\frak a }
\nc \cf{\frak c }
\nc \sfr{\frak s} 
\nc \ufr{\frak u} 
\nc \g{\frak g} 
\nc \gA{\g_{\Ao}} 
\nc \lfr{\frak l}
\nc \afr{\frak a}
\nc \gfh{\hat {\frak g}}
\nc \gl{\frak { gl }}
\nc \Sl{\frak {sl}}
\nc \SU{\frak {SU}}
\nc{\Homf}{\frak{Hom}}

\newcommand{\on}{\operatorname}
\nc\hankel{\on {Hankel}}
\nc\row{\on {row\ }}
\nc\nullity{\on {nullity }}
\nc\col{\on {col\ }}
\nc\rowm{\on {Row \ }}
\nc\loc{\on {lc \ }}
\nc\nullo{\on {null\ }}
\nc\Nul{\on {Nul\ }}
\nc \Ann {\on {Ann }}
\nc \Ass {\on {Ass \ }}
\nc \Coker {\on {Coker}}
\nc \Co{\on C}
\nc \Homo{\on {Hom}}
\nc \Ker {\on {Ker}}
\nc \omod{\on {mod}}
\nc \No {\on N}
\nc \NN {\on {NN}}
\nc \NGo {\on {NG}}
\nc \Oo {\on O}
\nc \ch {\on {ch}}
\nc \rko {\on {rk}}
\nc \Sing {\on {Sing\ }}
\nc \Reg {\on {Reg}}
\nc \CoI {\on {CI}}
\nc \CoM {\on {CM}}
\nc \Gor {\on {Gor}}
\nc \Type {\on {Type}}
\nc \can {\on {can}}
\nc \Top {\on {T}}
\nc \Tr {\on {Tr}}
\nc \rel {\on {rel}}
\nc \tr {\on {tr}}
\nc \sgn {\on {sgn }}
\nc \trdeg {\on {tr.deg}}
\nc \codim {\on {codim }}
\nc \coht {\on {coht}}
\nc \divo {\on {div \ }}
\nc \coh {\on {coh}}
\nc \Clo {\on {Cl}}
\nc \embdim{\on {embdim}}
\nc \embcodim{\on {embcodim \ }}
\nc \qcoh {\on {qcoh}}
\nc \grad {\on {grad}\ }
\nc \grade {\on {grade}}
\nc \hto {\on {ht}}
\nc \depth {\on {depth}}
\nc \prof {\on {prof}}
\nc \reso{\on {res}}
\nc \ind{\on {ind}}
\nc \prodo{\on {prod}}
\nc \coind{\on {coind}}
\nc \Con{\on {Con}}
\nc \Crit{\on {Crit}}
\nc \Der{\on {Der}}
\nc \Char{\on {Char}}
\nc \Ch{\on {Ch}}

\nc \Ext{\on {Ext}}
\nc \Eo{\on {E}}
\nc \End{\on {End}}
\nc \ad{\on {ad}}
\nc \Ad{\on {Ad}}
\nc \gr{\on {gr}}
\nc \Fo{\on {F}}
\nc \Gr{\on {Gr}}
\nc \Go{\on {G}}
\nc \GFo{\on {GF}}
\nc \Glo{\on {Gl}}
\nc \Ho{\on {H}}
\nc \CMo{\on {\CM}}
\nc \SCM{\on {SCM}}
\nc \hol{\on {hol}}
\nc{\sgd}{\on{sgd}}
\nc \supp{\on {supp}}
\nc \ssupp{\on {s-supp}}
\nc \singsupp{\on {singsupp}}
\nc \msupp{\on {msupp}}
\nc \spec{\on {spec}}
\nc \spano{\on {span }}
\nc \Span{\on {Span }}
\nc \Max{\on {Max}}
\nc \Min{\on {Min}}
\nc \Mod{\on {Mod}}
\nc \Rad {\on {Rad}}
\nc \rad {\on {rad}}
\nc \rank {\on {rank}}
\nc \range {\on {range}}
\nc \Slo{\on {SL}}
\nc \soc {\on {soc}}
\nc \Irr {\on {Irr}}
\nc \Imo {\on {Im}}
\nc \SSo{\on {SS}}
\nc \lub{\on {lub}}
\nc \gldim{\on {gl.d.}}
\nc \pdo{\on {p.d.}} 
\nc \ido{\on {i.d.}} 
\nc \dSSo{\dot {\SSo}}
\nc \So{\on S}
\nc \Io{\on I}
\nc \Jo{\on J}
\nc \jo{\on j}
\nc \Ko{\on K}
\nc \PBW{\Ac_{PBW}}
\nc \Ro{\on R}
\nc \To{\on T}
\nc \Ao{\on A}

\nc \Do{{\on D}}
\nc \Bo{\on B}
\nc \Po{\on P}
\nc \Qo{\on Q}
\nc \Zo{\on Z}
\nc \U{\on U}
\nc \wt{\on {wt}}
\nc \Uh{\hat {\U}}
\nc \T{\on T}
\nc \Lo{\on L}
\nc{\dop}{\on d}
\nc{\eo}{\on e}
\nc{\ado}{\on{ad}}
\nc{\Tot}{\on{Tot}}
\nc{\Aut}{\on{Aut}}
\nc{\sinc}{\on {sinc}}

\nc{\overrightleftarrows}[2]{\overset{#1}{\underset{#2}{\rightleftarrows}}}

\nc{\CCF}{\cal{CF}}

\nc{\CDF}{\cal{DF}}
\nc{\CHC}{\check{\cal C}}

\nc{\Cone}{\on{Cone}}
\nc{\dec}{\on{dec}}

\nc{\Diff}{\on{Diff}}

\nc{\dirlim}{\underset{\to}{\on{lim}}}
\nc{\dpar}{\partial}
\nc{\GL}{\on{GL}}
\nc{\CGr}{\cal{G}r}

\nc{\pr}{\on{pr}}
\nc{\semid}{|\!\!\!\times}
\nc{\Hom}{\on{Hom}}
\nc \RHom{\on {RHom}}

\nc \Proj{\mathrm {Proj\ }}
\nc \proj{\mathrm {proj}}
\nc{\Id}{\on{Id}}
\nc{\id}{\on{id}}
\nc{\Ima}{\on{Im}}
\nc{\invtimes}{\underset{\gets}{\otimes}}
\nc{\invlim}{\underset{\gets}{\on{lim}}}
\nc{\Lie}{\on{Lie}}

\nc{\re}{\on{Re }}
\nc{\Pic}{\on{Pic }}
\nc{\LPic}{\on{LPic }}
\nc{\Sch}{\on{Sch}}
\nc{\Sh}{\on{Sh}}
\nc{\Set}{\on{Set}}
\nc{\spo}{\on{sp\  }}
\nc{\Spec}{\on{Spec}}
\nc{\mSpec}{\on{mSpec}}
\nc{\Specb}{\bold {Spec}}
\nc{\Projb}{\bold {Proj}}
\nc{\Specan}{\on{Specan}}
\nc{\Spo}{\on{Sp}}
\nc{\Spf}{\on{Spf}}
\nc{\sym}{\on{sym}}
\nc{\symm}{\on{symm}}
\nc{\rop}{\on{r}}
\nc{\Td}{\on{Td}}
\nc{\Tor}{\on{Tor}}

\nc{\Artin}{\cal{A}rtin}
\nc{\Dgcoalg}{\cal{D}gcoalg}
\nc{\Dglie}{\cal{D}glie}
\nc{\Ens}{\cal{E}ns}
\nc{\Fsch}{\cal{F}sch}
\nc{\Groupoids}{\cal{G}roupoids}
\nc{\Holie}{\cal{H}olie}
\nc{\Mor}{\cal{M}or}

\nc{\CF}{\ensuremath{\cal{F}}}
\nc \Kc{\ensuremath{\cal K}}
\nc \Lc{{\ensuremath{\cal L}}}
\nc \lcc{{\mathcal l}} 
\nc \CC{{\ensuremath{\cal C}}} 
\nc \Cc{{\ensuremath {\cal C}}}
\nc \Pc{{\ensuremath{\cal P}}}
\nc \Dc{\ensuremath{\mathcal D}}
\nc \Ac{{\ensuremath{\cal A}}} 
\nc \Bc{{\ensuremath{\cal B}}}
\nc \Ec{{\ensuremath{\cal E}}}
\nc \Fc{{\ensuremath{\cal F}}}
\nc \Mcc{{\ensuremath{\cal M}}} 
\nc \hM{\hat{\Mcc}} 
\nc \bM{\bar {\Mcc}} 
\nc\hbM{\hat{\bar \Mcc}}  
\nc \Nc{{\ensuremath{\cal N}}}
\nc \Hc{{\ensuremath{\cal H}}} 
\nc \Ic{{\ensuremath{\cal I}}} 
\nc \Oc{\ensuremath{{\cal O}}}
\nc \Och{\hat{\cal O}} 
\nc \Sc{{\ensuremath{{\cal S}}}}
\nc \Tc{\ensuremath{{\cal T}}} 
\nc \Vc{{\ensuremath{{\cal V}}}} 
\nc{\CA}{{\ensuremath{{\cal A}}}}
\nc{\CB}{{\ensuremath{{\cal B}}}}
\nc{\C}{{\ensuremath{{\cal F}}}}
\nc{\Gc}{{\ensuremath{{\cal G}}}}
\nc{\CH}{\ensuremath{\mathcal H}}
\nc{\CI}{{\ensuremath{{\cal I}}}}
\nc{\CM}{{\ensuremath{{\cal M}}}}
\nc{\CN}{{\ensuremath{{\cal N}}}}
\nc{\CO}{{\ensuremath{{\cal O}}}}
\nc{\Rc}{{\ensuremath{{\cal R}}}}
\nc{\CT}{{\ensuremath{\mathcal T}}}
\nc{\CU}{\ensuremath{{\cal U}}}
\nc{\CV}{\ensuremath{{\cal V}}}
\nc{\CZ}{\ensuremath{{\cal Z}}}
\nc{\Homc}{\ensuremath{{\cal {Hom}}}}

\nc{\Tab}{\ensuremath{{\mbox{Tab}}}}

\nc{\fa}{\frak{a}}
\nc{\fA}{\frak{A}}
\nc{\fg}{\frak{g}}
\nc{\fh}{\frak{h}}
\nc{\fI}{\frak{I}}
\nc{\fK}{\frak{K}}
\nc{\fm}{\frak{m}}
\nc{\fP}{\frak{P}}
\nc{\fS}{\frak{S}}
\nc{\ft}{\frak{t}}
\nc{\fX}{\frak{X}}
\nc{\fY}{\frak{Y}}

\nc{\bF}{\bar{F}}
\nc{\bCP}{\bar{\cal{P}}}
\nc{\bm}{\mbox{\bf{m}}}
\nc{\bT}{\mbox{\bf{T}}}
\nc{\hB}{\hat{B}}
\nc{\hC}{\hat{C}}
\nc{\hP}{\hat{P}}
\nc{\htest}{\hat P}

\nc{\bS}{\mbox{\bf{S}}}

\nc{\nen}{\newenvironment}
\nc{\ol}{\overline}
\nc{\ul}{\underline}
\nc{\ra}{\to}
\nc{\lla}{\longleftarrow}
\nc{\lra}{\longrightarrow}
\nc{\Lra}{\Longrightarrow}
\nc{\Lla}{\Longleftarrow}
\nc{\Llra}{\Longleftrightarrow}
\nc{\hra}{\hookrightarrow}
\nc{\iso}{\overset{\sim}{\lra}}

\nc{\dsize}{\displaystyle}
\nc{\sst}{\scriptstyle}
\nc{\tsize}{\textstyle}

\nen{exa}[1]{\label{#1}{\bf Example.\ } }{}

\nen{rem}[1]{\label{#1}{\em Remark.\ } }{}
\nen{exer}[1]{\label{#1}{\em Exercise.\ } }{}

\newcommand{\la}{\lambda}

\newcommand {\bC} {\mathbb C}

\newcommand{\bN}{\mathbb N}
\newcommand {\D} {\mathcal D}

\newcommand{\DD}{{\mathcal{D}}_V}

\newcommand{\OO}{{\mathcal{O}}_U}

\begin{document}

\title{Higher Specht Polynomials and Modules over the Weyl algebra}

\author{ Ibrahim Nonkan\'e,
 L\'eonard  Todjihounde}

\address{Departement d'\'economie et de math\'ematiques appliqu\'ees, IUFIC, Universit\'e Thomas Sankara, Burkina faso}
\email{inonkane@univ-ouaga2.bf}

\address{Institut de Math\'e,matiques et de Sciences Physiques (IMSP), Universit\'e d'Abomey-Calavi, Benin}
\email{leonardt@imsp-uac.org}

\maketitle
\begin{abstract}
In this paper, we study an irreducible  decomposition structure of the $\Dc$-module direct image $\pi_+(\Oc_{ \bC^n})$ 
for  the finite map $\pi: \bC^n  \to  \bC^n/ ({\Sc_{n_1}\times \cdots \times \Sc_{n_r}}).$ 
We explicitly construct the simple component of $\pi_+(\Oc_{\bC^n})$ by providing their generators  and  their multiplicities.
Using an equivalence categories and the higher Specht polynomials, we describe a $\D$-module decomposition of the  of the polynomial ring localized  at the discriminant of $\pi$. Furthermore we study the action invariants differential operators on the higher Specht polynomials. 
\\\\
\textbf{keywords}:\ {Direct image, Direct product, Group representation theory, Higher Specht polynomials, Partitions,  Primitive idempotents, Semisimplicity, Symmetric group , Young diagram.}
\end{abstract}

\footnote{{\bf Acknowledgements} This  work was done while the first author was visiting IMSP at Benin,  under the  {\it Staff Exchange} program of the German  Office of Academic Exchange  (DAAD).
He warmly thanks the DAAD for the financial support.
}
\section{Introduction} 
 It is well-known that the ring $\Oc_X:={\mathbb C}[x_1,...,x_n]$  of the polynomials in $n$ indeterminates is a simple module over the Wyel algebra $\Dc_X$ associated with the $\Oc_X.$ 
The direct image of a simple module under a proper map $\pi$ is semisimple by the Kashiwara's decomposition Theorem \cite{Cataldo}. 
The simplest case is when the map $\pi: X \to Y$ is finite, in which case it is easy to give an elementary and wholly algebraic proof, using essentially the generic correspondence with the differential Galois group, which equals the ordinary Galois  group $G$. The irreducible submodule of the direct image are in one-to-one correspondence with the irreducible representations of $G$ (see \cite{IBK}).  
The goal of this paper is to find the simple component of the direct image $\pi_+(\Oc_X)$ of the polynomial ring $\Oc_X$ as a $\Dc$-module  under the map $  \pi: \spec \Oc_X \to \spec \Oc_Y$ where $ \Oc_Y= \Oc_X^{\Sc_{n_1}\times \cdots \times \Sc_{n_r}};$ the ring of invariant polynomials under the action of ${\Sc_{n_1}\times \cdots \times \Sc_{n_r}}.$  
We describe The generators of the simple components of $\pi_+(\Oc_X)$ and their multiplicities as in \cite{IBK}.
 We then  give the decomposition of the structure of $\pi_+(\Oc_X)$ by  giving a basis generated by the higher  Specht polynomials.  This proof uses the fact  that the irreducible $\D$-submodules of $\pi_+(\Oc_X)$ are in one-to-one correspondence with irreducible representations of $ {\Sc_{n_1}\times \cdots \times \Sc_{n_r}}$. \\
 Secondly we  elaborate a $\D$-module decomposition of the polynomial ring localized  at the discriminant of $\pi.$
 Finally we describe the action invariants differential operators on  higher Specht polynomials
The higher Specht polynomials (introduced combinatorially by several authors \cite{Tera-Yama}, \cite{Ariki}), are adapted to the $\D$-module structure.\\
This  paper  generalizes results on modules over the Weyl algebra appeared in \cite{IBK} and \cite{RTG}.
The case $r=2$ have been presented at 10th International Conference on Mathematical Modeling in Physical Sciences in order to describe the action of the rational quantum Calogero-Moser system on polynomials  \cite{IB}.

\section{Preliminaries}

\subsection{Direct image} 
We briefly recall the definition of the direct image of a $\D$-module \cite{CSC}.\\
Let $K$ be a field of characteristic zero,  put $X=K^n$. The polynomial ring $K[x_1,\ldots, x_n]$ will be denoted by $K[X];$ and the Weyl algebra generated by $x_i$'s and $\frac{\partial}{\partial x_i}$'s by $\D_X$. The $n$-tuple $(x_1, \ldots,x_n)$ will be denoted by $X$. Similar conventions will holds for $Y=K^m$, with polynomial ring $K[Y]$ and Weyl algebra $\D_Y$.\\

 Let $\pi  : X \to Y$ be a polynomial map, with $\pi=(\pi_1,\ldots,\pi_m)$. Let $M$ be a left $\D_Y$-module. The inverse image of $M$ under the map  $\pi$ is $\pi^+( M)= K[X] \otimes_{K[Y]}M.$ This is a $K[X]$-module. It becomes a $\D_X$-module with $\partial_{x_i}$ acting according to the formula
$$\frac{ \partial }{\partial x_i} (h \otimes u) = \frac{\partial h}{\partial{x_i}}\otimes u + \sum_{j=1}^m \frac{\partial  \pi_j}{\partial x_i} \otimes \frac{\partial }{\partial y_j} u, \ h \in K[X], u \in M.$$
Since $\D_Y \otimes_{\D_Y} M \cong M,$ we have that 
$$ \pi^+ (M) \cong K[X] \otimes_{K[Y]} \D_{Y} \otimes_{\D_Y}M = \pi^+(K[Y])  \otimes_{\D_Y}M.$$
Writing $D_{X \to Y}$ for $\pi^+( K[Y])$, on has that $\pi^+(M)= \D_{X \to Y}  \otimes_{\D_Y}M.$
Note that $\D_{X\to Y}$ is $\D_X$- $\D_Y$-bimodule.

Let $N$ be a right $D_X$-module. The tensor product 
$$\pi_+(N) =N \otimes_{\D_X} \D_{X \to Y}$$ is a right $\Dc_Y$-module, which is called the {\it  direct image} of $N$ under the polynomial map $\pi$. 
Let us consider the {\it standard transposition} $ \tau : \D_X \to \D_X$ defined by $ \tau (h \partial ^{\alpha} )= (-1)^{|\alpha|} \partial^{\alpha} h,$ where $h \in K[X]$ and $\alpha \in \bN^n.$ If $N$ is a right $\D_X$-module then we define a left $\D_X$-module $N^t$ as follows. As an abelian group, $N^t=N.$ If $a \in \D_X$ and $ u \in N^t$ then the left action of $a$ on $u$ is defined by $a \star u= u \tau(a).$ Using the standard transposition for $\D_Y$ and $\D_X$, put $D_{ Y\leftarrow X} = (D_{ X \to Y} )^t$, this is a $\D_Y$-$\D_X$-bimodule. Let $M$ be a left $\D_X$-module. The direct image of $M$ under $\pi$ is defined by the formula
$$ \pi_+(M)= D_{ Y \leftarrow X} \otimes_{\D_X} M.$$ It is clear that $\pi_+(M)$ is  a $\D_Y$-module.\\
The following is the Kashiwara decomposition theorem
\begin{theorem} \cite{Cataldo} 
Let $\pi: X\to Y$ be a polynomial map. If $M$ is a a simple (holonomic) module over $\D_X$. Then $\pi_+ (M)$ is a semisimple $\D_Y$-module. we have 
$$ \pi_+(M)= \oplus  M_i^{\alpha_i},$$ where the $M_i$ are inequivalent irreducible $\D_Y$-submodules.
\end{theorem}

\subsection {Higher Specht polynomials}
In this subsection we recall some notions about irreducibles of representations of product of symmetric groups (see \cite{Ariki} for more details). \\

Let $\Oc_X$ be the algebra of polynomials of $n$ variables $x_1,\ldots,x_n$ with complex coefficients, on which the symmetric group $\Sc_n$ acts by permutation of variables: 
$$( \sigma f)(x_1,\ldots,x_n)=f( x_{\sigma(1)}, \ldots, x_{\sigma(n)}), \sigma \in \Sc_n , f \in \Oc_X.$$
Let $n_1,\ldots,n_r$ be natural numbers such that $n=\sum_{i=1}^r n_i.$ Then the product of symmetric groups $\Sc_{n_1} \times \cdots \times \Sc_{n_r}$ is naturally embedded in $\Sc_n.$\\
A partition $\lambda$ is a non-increasing finite sequence of positive  integers $\lambda \geq \cdots \geq \lambda_l > 0.$ We write $\lambda \vdash n$ when  $\sum_{i=1}^l \lambda_i=n$, and $n$ is called the size of $\lambda.$ To every partition  corresponds  a Young diagram \cite {Sagan}. 
Let  $r$ be a positive integer and $\lambda = (\lambda^1,\ldots, \lambda^r)$ a $r$-tuple of partitions (Young diagrams), with $\lambda^1 \vdash n_1,\ldots, \lambda^r \vdash n_r$ ,  $\lambda$ is called an $r$-diagram. The sequence $(n_1,\ldots, n_r)$ is called the type $\lambda$ and denoted by $type(\lambda)$ and $n$ called the size of $\lambda$. The irreducible representations of $\Sc_{n_1} \times \cdots \times \Sc_{n_r}$ are indexed by the set of $r$-diagrams of type $(n_1,\ldots,n_r).$ By filling each "box" with a non-negative integer, we obtain an $r$-tableau from an $r$-diagram. The original $r$-diagram is called the shape of the $r$-diagram. An $r$-tableau $ T= (T^1,\dots, T^r)$ is said to be standard if the written sequence on each column and row of $T^i\ (1 \leq i \leq r)$ is strictly increasing, and each number from 1 to $n$ appears exactly once. The set of all standard $r$-tableaux of shape $\lambda$ is denoted by $ST(\lambda)$. 
A standard $r$-tableau $ T= (T^1,\ldots,T^r)$ is said to be natural if and only if the set of numbers written in $T^{i}$ is $\{ n_1+\cdots +n_{i-1}+1, \ldots, n_1+\ldots+n_{i} \}$. The set of natural standard $r$-tableaux of shape $\lambda$ is denoted by $NST(\lambda)$.\\

For a standard $r$-tableau $T$, we associate a word  $w(T)$ in the following way.  First we read each column of the tableau $T^1$ from the bottom to the top starting from the left. We continue this procedure for the tableau $T^2$ and so on.
We define the index $i(w(T))$ of $w(T)$ as follows. The number 1 in the word $w(T)$ has index 0. If $k$ in the word has index $p$, then $k+1$ has index $p$ or $p+1$ according as it lies to the right or the left of $k$.
Assigning the index $i(w(T))$ to the corresponding of $w(T)$ to the corresponding box, we get a new $r$-tableau $i(T)$ which is called the index $r$-tableau of $T$.

\subsection{Example} 
 For   $n=8, r=2 \  n_1=5 ,\  n_2=3,\  \lambda^1=(3,2),\  \lambda^2=(2,1) $  

$$  \lambda = \begin{pmatrix} \begin{Young}
  &  &  \cr
&  \cr
\end{Young},
\ \begin{Young}
  &  \cr
  \cr
\end{Young}
\end{pmatrix}
$$
with
$$  T =\begin{pmatrix}  \begin{Young}
 1 &4  &6  \cr
 2&7 \cr
\end{Young},
\ \begin{Young}
 3 & 8 \cr
 5 \cr
\end{Young}
\end{pmatrix}
$$
and

$$  i (T)= \begin{pmatrix} \begin{Young}
 $0$ & $2$ & $3$ \cr
$1$ & $4$ \cr
\end{Young},\ \begin{Young}
$1$ & $4$   \cr
$2$   \cr
\end{Young}  
\end{pmatrix}
.$$
For words $u=(u_1,\ldots,u_n)$ and $v=(v_1,\ldots,v_n)$, we define $\displaystyle x_v^u= x_{v_1}^{u_1} \cdots x_{v_n}^{u_n}$. For standard $2$-tableaux $S, T,$ we define $x_{T}^{i(S)}= x_{w(T)}^{i(w(S))}$.\\

Let $ T=(T^1,\ldots, T^r)$ be a standard $r$-tableau of shape $\lambda$. For each component $T^i\ (1\leq i \leq r),$ The Young symmetrizer $\eb_{T^i}$ of $T^i$ is defined by
\begin{equation}
 \eb_{T^i}= \frac{f^{\la_i}}{ {n_i}!} \sum_{\sigma \in R(T^i)\ \tau \in C(T^i)} \sgn (\tau) \tau \sigma \in \bC[\Sc_{n_i}],
 \end{equation}
  where $f^{\la_i}$ is the number of standard tableau of shape $\lambda^i$,  $R(T^i)$ and $C(T^i)$ are the \textit{row-stabilizer} and \textit{colomn-stabilizer} of $T^i,$ respectively.
We set \begin{equation} \eb_{T}= \eb_{T^1} \cdots \eb_{T^r}. \end{equation}  For $T, S \in ST(\lambda)$, Ariki, Terasoma and Yamada have defined the higher Specht polynomial for $(T, S)$ in \cite{Ariki} by
\begin{equation}
F_{T}^{S} = F_{T}^{S}(x_1,\ldots,x_n)= \eb_{T} (x_{T}^{i(S)}).
\end{equation}

Let $\Pc_{r,n}$ be the set $r$-tuples of Young diagrams $\lambda$ of type $n$
\begin{theorem} \cite{Ariki}
 For an $r$-diagram $\lambda$ of type $(n_1,\ldots,n_r)$ and  $S \in ST(\lambda)$, the set  $\{ F_{T}^{S} | T \in NST(\lambda) \}$ forms a $\bC$-basis of a $\bC[\Sc_{n_1} \times \cdots \times  \Sc_{n_r}]$-submodule denoted by $V^{S}(\lambda)$, which affords  an irreducible representation of  $\Sc_{n_1} \times \cdots \times  \Sc_{n_r}$ corresponding to $\lambda$. All the other irreducible representation of $\Sc_{n_1} \times \cdots \times  \Sc_{n_r}$ are obtained by same procedure.
\end{theorem}

%

\section{Decomposition theorem}
We are interested in studying the decomposition structure of $\pi_+(M)$, where $M=\Oc_X,\ \pi :X= \spec(\Oc_X)\to Y=\spec(\Oc_X^{ \Sc_{n_1} \times \cdots \times \Sc_{n_r}}) $. Since $\Oc_X$ is a holonomic $\D_X$-module \cite[Chapter 10]{CSC}, $\pi_+(\Oc_X)$ is a semisimple $\D_Y$-module by the Kashiwara decomposition theorem. We  construct the simple components of  $\pi_+(\Oc_X)$ and provide their multiplicities.  Let us recall some useful facts from \cite{IBK}.\\

Let $\Delta:= Jac( (\pi)$ be the Jacobian of $\pi$, $\Delta^2$ the discriminant of $\pi$ we denote the complement of the branch locus and the discriminant by $U$ and $V$, respectively.
Assume now that $U,V$ are such that the respective canonical modules are generated by volume forms $dx$, and $dy$,  related by $dx=\Delta dy$, where  $\Delta$ is the Jacobian of $\pi$. 
 \begin{proposition}  
\label{prop:genericstudy}
\begin{itemize}
 \item[(i)] There is an isomorphism of $\D_V$-modules $$T:\pi_+(O_U)\cong O_U,\quad r(  dy^{-1}\otimes dx)\mapsto r \Delta ^{-1}.$$
\item[(ii)]
  $T(\pi_+(O_X))$ is isomorphic as a $\D_Y$-module to $\pi_+(\Oc_X)$.
\end{itemize}
\end{proposition}
\begin{proof}
See \cite[Lemma 2.3]{IBK}.
\end{proof}
It is more convenient to study $T(\pi_+(O_X))\cong \pi_+(O_X)$, as a submodule of $O_U$, than using the definition of $\pi_+(O_X)$. Therefore to reach our goal, we will first  study the decomposition of $O_U$ into irreducible components as a $\D_V$-module.\\
The following proposition enables us to reduce the study of the decomposition factors of $\pi_+(O_X)$ 
to the behavior of the direct image over the complement to the branch locus, or even over the generic point. Let $j:U\hookrightarrow X$  and $i:V\hookrightarrow Y$ be the inclusions.

\begin{proposition}  \label{prop:etale.semisimple2}
Let $\pi: X\to Y$ be a finite map. Then
\begin{itemize}
\item[(i)]
 $\pi_+(O_X)$ is semi-simple as a $\D_Y$-module.
\item[(ii)] If $\pi_+(O_X)=\oplus M_i,\ i\in I$ is a decomposition into simple
(non-zero) $D_Y$-modules, then $\pi_+(O_U)=\oplus i^+(M_i),\ i\in I$, is a decomposition of  $\pi_+(O_U)$ into simple (non-zero) $\Dc_{V}$-modules.
\end{itemize}
\end{proposition} 
\begin{proof}
See \cite[Proposition 2.8]{IBK}.
\end{proof}
\subsection{Notation}
Let $\displaystyle \Dc_X:=\bC \langle x_1, \ldots, x_n, \frac{\partial}{\partial x_1}, \ldots, \frac{\partial}{\partial x_n} \rangle$ be the Weyl algebra associated with the polynomial ring $\Oc_X$, and $\Oc_Y:=\Oc_X^{\Sc_{n_1} \times \cdots \times \Sc_{n_r}}= \bC[y_1,\ldots,y_n]$ be the ring of invariant polynomials under the group $\Sc_{n_1} \times \cdots \times \Sc_{n_r}$. We denote by $\Dc_Y= \bC \langle y_1, \ldots, y_n, \frac{\partial}{\partial y_1}, \ldots, \frac{\partial}{\partial y_n} \rangle $ the Weyl algebra associated with  $\Oc_Y$. 
We have 
 $\displaystyle \OO =\bC [x_1,\ldots,x_n, \Delta^{-1}],\  \displaystyle \Oc_{V} =\bC [y_1,\ldots,y_n, \Delta^{-2}],$ and $$  \DD = \bC \langle y_1, \ldots, y_n, \frac{\partial}{\partial y_1}, \ldots, \frac{\partial}{\partial y_n}, \Delta^{-2} \rangle.$$
 We adopt the following notations  for $i =1,\ldots, r$
 $$ \Delta_i = \prod_{  n_1+ \cdots + n_{i-1} +1 \leq i<j \leq { n_1+ \cdots + n_i}} (x_i-x_j)$$
$$\displaystyle \Oc_{X_i} :=\bC [x_{ n_1+ \cdots + n_{i-1} +1}, \ldots, x_{ n_1+ \cdots + n_i} \Delta_i^{-1}],$$
$$\displaystyle \Oc_{Y_i} :=\bC [y_{ n_1+ \cdots + n_{i-1} +1}, \ldots, y_{ n_1+ \cdots + n_i} \Delta_i^{-2}],\ \mbox{and}$$
$$   \Dc_{Y_i}:= \bC \langle y_{ n_1+ \cdots + n_{i-1} +1}, \ldots, y_{ n_1+ \cdots + n_i}, \frac{\partial}{\partial y_{ n_1+ \cdots + n_{i-1} +1}}, \ldots, \frac{\partial}{\partial y_{ n_1+ \cdots + n_i}}, \Delta_i^{-2} \rangle.$$

Then  we have $ \OO=\Oc_{X_1} \otimes \cdots \otimes  \Oc_{X_r},$  $ \Oc_V=\Oc_{Y_1} \otimes \cdots \otimes  \Oc_{Y_r},$ $\DD = \Dc_{Y_1} \otimes \cdots \otimes  \Dc_{Y_r}$ and $\displaystyle \Delta= \prod \Delta_i.$

For $M_i$ a $\Dc_{Y_i}$-module $i=1,\ldots,r$, we make $M_1\otimes \cdots \otimes M_r$ into a $\DD $-module by setting
\begin{equation} (D_1 \otimes \cdots \otimes  D_r) (m_1 \otimes \cdots \otimes m_r)= D_1 m_1 \otimes \cdots \otimes D_r m_r \end{equation}  
for $D_i \in \Dc_{Y_i}$ and $m_i \in \Oc_{X_i}, i=1\ldots,r.$
\begin{lemma}
 $\OO$ is a $\DD$-module. 
\end{lemma}
\begin{proof}
Since by \cite[Lemma 3.1]{DOS} $\Oc_{X_i}$ is a $\Dc_{Y_i}$-module for $i=1, \ldots, r$, it clearly  follows from equality (3.1).
\end{proof}

Let $V_i $ be a $\bC[\Sc_{n_i}]$-module $(i=1,\ldots,r)$, we define the action $\bC[\Sc_{n_1} \times \cdots \times \Sc_{n_2}]$ on   $V_1\otimes \cdots \otimes V_r$  by
$$(s_1 \times \cdots \times s_r) (v_1 \otimes \cdots  \otimes  v_r)= s_1 v_1 \otimes  \cdots \otimes s_r v_r,$$ for $s_i \in \Sc_{n_i}, v_i \in V_i, 1,\ldots ,r.$
This makes $V_1\otimes \cdots \otimes  V_r$  into $\bC[\Sc_{n_1} \times \cdots \times \Sc_{n_r}].$ If $V_i $ is an irreducible $\bC[\Sc_{n_i}]$-module $(i=1,\ldots, r)$, then $V_1\otimes \cdots \otimes V_r$ is an irreducible  $\bC[\Sc_{n_1} \times \cdots \times \Sc_{n_r}]$-module and all irreducible  $\bC[\Sc_{n_1} \times \cdots \times \Sc_{n_r}]$-modules have this form\cite[Chapter IV \S 27]{Curtis Reiner}.\\
We   know  by  \cite[Proposition 3.2]{End-Muk} that: 
\begin{equation}  \bC[\Sc_{n_1} \times \cdots \times \Sc_{n_r}]  \cong \bC[\Sc_{n_1}]  \otimes \cdots \otimes \bC[\Sc_{n_r}].
\end{equation} A basis over $\bC$ of $ \bC[\Sc_{n_1} \times \cdots \times \Sc_{n_r}]$ is given 
$\C=\{ F_{T}^{S}; S \in ST(\lambda), T \in NST(\lambda) | \lambda \in \Pc_{r,n} \}$. For every couple $(\la, S) \in \Pc_{r,n} \times ST(\la)$ corresponds an irreducible  $\Sc_{n_1} \times \cdots \times \Sc_{n_r}$-representation $V^S(\la).$  For $i=1,\ldots, r$ a basis over $\bC$ of $\bC[\Sc_{n_i}]$ is given by $\C_i=\{ F_{T^i}^{S^i}; S^i , T^i \in ST(\lambda) | \lambda \vdash n_i \}.$ For every couple $(\la_i, S^i) \in \Pc_{1,n_i} \times ST(\la_i)$ corresponds an irreducible  $\Sc_{n_i}$-representation $V^{S^i}(\la_i),\ i=1\ldots, r$,  \cite{Tera-Yama}. 

\begin{lemma}[Identification map]
For $T=(T^1 \ldots, T^r) \in NST(\la)$ and $S=(S^1, \ldots, S^r) \in ST(\la)$, define the linear map $\varphi: \bC[\Sc_{n_1} \times \cdots \times \Sc_{n_r}] \to  \bC[\Sc_{n_1}]  \otimes \cdots \otimes \bC[\Sc_{n_r}]$ by $\varphi (F_T^S)= F_{T^1}^{S^1}\otimes \cdots \otimes F_{T^r}^{S^r}.$ Then $\varphi$ is a $\bC[\Sc_{n_1} \times \cdots \times \Sc_{n_r}]$-isomorphism.
Moreover we have  $$\varphi(V^S(\la))= V^{S^1}(\la_1) \otimes \cdots \otimes V^{S^r}(\la_r).$$
\end{lemma}
\begin{proof}
It obvious that $\varphi$ is a $\bC[\Sc_{n_1} \times \cdots \times \Sc_{n_r}]$ which is a bijection.
\end{proof}
From now we will use the map $\varphi$ in Lemma 3.4 to identify elements of $\bC[\Sc_{n_1} \times \cdots \times \Sc_{n_r}] $ with elements of $\bC[ \Sc_{n_1}] \otimes \cdots \otimes \bC[\Sc_{n_r} ].$

\subsection{Simple components and their multiplicities}
Let  $\lambda \in \Pc_{r,n}$ be an $r$-diagram of size $n$ and  $ T \in NST(\Lambda)$ a natural standard tableau and let the  $\eb_{T}$ be as in (2,2). The element $\eb_T$ is a primitive idempotents $\bC[\Sc_{n_1} \times \cdots \times \Sc_{n_r}]$ and each primitive idempotent of $\bC[\Sc_{n_1} \times \cdots \times \Sc_{n_r}]$ is associated with a natural standard tableau \cite[chapter V,$\S$ 10]{Weyl}. $\{ \eb_T ; T\in NST(\lambda), \lambda \in \Pc_{r,n}  \}$ is the complete list of all primitive idempotents of $\bC[\Sc_{n_1} \times \cdots \times \Sc_{n_r}].$ \\

For $i =1,\ldots, r$ let $\lambda_i \vdash n_i$;  the canonical standard tableau $S_0^i$ of shape $\lambda_i$  is the unique  $\lambda_i$-tableau  whose  cells are numbered from the left to the right in successive rows, starting from the top. Let $T^i$ be a $\lambda_i$-standard tableau, we denote by $F_{T^i}$ the ordinary Specht polynomial associated with $T^i$ \cite{Peel}. Then  the higher Specht polynomial  $F_{T^i}^{S_0^i}$ is proportional to the Specht polynomial $F_{T^i}$  \cite{Tera-Yama} .
The following theorem is the analog  of \cite[Theorem 3]{RTG} for product of symmetric groups.
\begin{theorem}
Let   $\lambda \in \Pc_{r,n}$ be an $r$-diagram of size $n$, $T \in NST(\lambda)$ a natural standard tableau of shape $\lambda$, and  $\eb_{T}$ is the  primitive idempotent associated with $T$. Then we have :
\begin{enumerate}
\item $ \eb_{T}\OO$ is a nontrivial  $\DD$-submodule of $\OO,$
\item The $\DD$-module $\eb_{T}\OO$ is simple,
\item There exist a  $ S  \in ST(\lambda)$ and a higher  Specht polynomial $F_{T}^{S}$ for $(T, S)$   such that $\eb_{T} \OO= \DD F_{T}^{S}.$
\end{enumerate}
\end{theorem}

\begin{proof}
Let $\lambda \in \Pc_{r,n}$ be an $r$-diagram of size $n$ and  $T \in NST(\lambda)$ , There exist $n_1, \ldots, n_r \in \mathbb N, \lambda^i \vdash n_i,$ and  $T_i\in ST(\lambda_i),\ i=1,\ldots,r$ such that $$ \sum n_i=n,\ \lambda= (\lambda^1, \ldots, \lambda^r) \ \mbox{and}\  \eb_{T}= \eb_{T^1} \cdots \eb_{T^r}.$$ 
\begin{enumerate}
\item We know that $e_{T^i}$ is an primitive idempotent for $\Cb[\Sc_{n_i}], i=1,\ldots,r$ and    $\eb_{T} \OO = (\eb_{T^1} \cdots \eb_{T^r} )\OO= \eb_{T^1} \Oc_{X_1} \otimes \cdots \otimes \eb_{T^r} \Oc_{X_r}.$ By \cite[Theorem 3]{RTG},  $\eb_{T^i} \Oc_{X_i}$  is a nontrivial  $\Dc_{Y_i}$ module for $i=1,\ldots,r$. Hence $\eb_{T} \OO$ is a nontrivial module over $\DD.$
\item Since $e_{T^i}$  being a primitive idempotent for $\Cb[\Sc_{n_i}], i=1,\ldots, r$, by \cite[Theorem 3]{RTG}, we have that $\eb_{T^i} \Oc_{X_i}$ is a simple $\Dc_{Y_i}$-modules for $i=1,\ldots,r$. Then $\eb_{T^1} \Oc_{X_1} \otimes \cdots \otimes \eb_{T^r} \Oc_{X_r}$ is an irreducible $\Dc_{Y_1} \otimes \cdots \otimes  \Dc_{Y_r}$-module. Hence $\eb_{T} \OO$ is a simple $\DD$-module. 
\item Let   $S_0^i$ be the canonical standard tableau of  shape $\lambda_i, i=1,\ldots, r$ and we know that  the higher Specht polynomial $F_{T^i}^{S_0^i},$  is proportional to the Specht polynomial $F_{T^i}$ of $T^i,\ i=1,\ldots, r.$ Then  By \cite[Theorem 3]{RTG} we have that  $\eb_{T^i} \Oc_{X_i}=  \Dc_{Y_i} F_{T^i}, i=1,\ldots,r$ so that, 
\begin{eqnarray*}
\eb_{T} \OO&=&\eb_{T^1}  \Oc_{X_1} \otimes \cdots \otimes  \eb_{T^r} \Oc_{X_r} \\
&=& \Dc_{Y_1} F_{T^1} \otimes \cdots \otimes    \Dc_{Y_r} F_{T^r} \  \mbox{by \cite[Theorem 3 (iii)]{RTG}} \\
&=&   \Dc_{Y_1} F_{T^1}^{S_0^1} \otimes \cdots \otimes    \Dc_{Y_r} F_{T^r}^{S_0^r} \\
&=& (\Dc_{Y_1} \otimes \cdots \otimes  \Dc_{Y_r} ) ( F_{T^1}^{S_0^1} \otimes \cdots \otimes   F_{T^r}^{S_0^r}) \\
&=& \DD F_{T}^{S_0} \  \  \mbox{by\ the identification\ map\ Lemma\ 3.3}
\end{eqnarray*}
$\mbox{where}\  T=(T^1, \ldots, T^r) \  \mbox{and}\  S_0=(S_0^1, \ldots, S_0^r).$
\end{enumerate}
\end{proof}
From now we adopt the following notation.
Let  $\lambda$ be an $r$ -diagram of size $n$, $ T \in NST(\lambda) $, put $F_T:=F_T^{S_0}$ where $S_0=(S_0^1, \ldots, S_0^r)$ so that $\eb_T \OO =\DD F_T.$ We denote by $F_{\lambda}$ the unique higher Specht polynomial $F_{S_0}^{S_0}.$

\begin{corollary}

With the above notations, $\eb_{T_1}\OO \cong_{\DD} \eb_{T_2} \OO$ if $T_1$ and $T_2$  have the same shape i.e. if there is a $r$-diagram $\lambda $ of size $n$ such that $ T_1, T_2 \in NST(\lambda ).$
\end{corollary}

\begin{proof}
Let $T_1, T_2 \in NST(\lambda)$ with $T_1= (T_1^1,\ldots,T_1^r)$ and $ T_2=(T_2^1, \ldots, T_2^r)$. Then $\eb_{T_1}= \eb_{T_1^1} \cdots \eb_{T_1^r}$ and $\eb_{T_2}= \eb_{T_2^1} \cdots \eb_{T_2^r}$, so that $ \eb_{T_k}\OO  \cong \eb_{T_k^1} \Oc_{X_1} \otimes \cdots \otimes  \eb_{T_k^r} \Oc_{X_r},$ $k=1,2.$ By \cite[Corollary 2]{RTG}, we know that 
$\eb_{T_1^i} \Oc_{X_i} \cong_{\Dc_{Y_i}} \eb_{T_2^i} \Oc_{X_i}$ if $T_1^i$ and $T_2^i$  have the same shape, $i=1,\ldots,r$. Hence $\eb_{T_1}\OO \cong_{\DD} \eb_{T_2} \OO$ if $T_1$ and $T_2$  have the same shape.
\end{proof}

\begin{proposition}
Let  $\lambda=(\la^1, \ldots, \la^r)$ be an $r$ -diagram of size $n$, $ T \in NST(\lambda) $, and $\eb_{T}$ the  primitive idempotent associated  with $T$ 
Then with the notation above, we have: 
\begin{enumerate}
\item  \begin{equation}
  \displaystyle \OO=  \bigoplus_{\lambda \in \Pc_{r,n}} \bigg( \bigoplus_{{ T\in NST(\lambda)}}  \DD F_{T} \bigg),
  \end{equation}
\item 
\begin{equation}
  \displaystyle \OO \cong  \bigoplus_{\lambda \in \Pc_{r,n}}    f^\lambda \DD F_{\lambda},
  \end{equation} 
  where $\displaystyle f^\la = \dim_{\bC} (V^{S_0}(\la)).$
 \end{enumerate}
 
\end{proposition}
\begin{proof}
\begin{enumerate}
\item
Since  $ \displaystyle 1= \sum_{ \lambda \in \Pc_{r,n}} \sum_{ T \in NST(\lambda)}  \eb_{T}$, we have   $\displaystyle\OO =\sum_{\lambda \in \Pc_{r,n}}\sum_{ T \in NST(\lambda)}  \eb_{T} \OO$. Let $m \in \eb_{T_1} \OO \cap \eb_{T_2} \OO$ with $T_1 \neq T_2$ so that $m = \eb_{T_1} m_1$ and $ m=\eb_{T_2} m_2.$ Then $\eb_{T_1} m= \eb_{T_1} \eb_{T_2} m_2 =0$, hence $m=0$. It is clear that $ \eb_{T_1} \OO \cap \eb_{T_2} \OO = \{0\}$ and  
$$  \OO = \bigoplus_{\lambda } \bigg( \bigoplus_{{T\in NST(\lambda)}} \eb_{T} \OO \bigg),$$
 where the $\eb_{T} \OO$ are simple $\DD$-modules. Since to each an $r$-tableau ${T}$ corresponds a higher Spect polynomial  $F_T$ such  that $\eb_{T} \OO= \DD F_{T}$ then $\displaystyle \OO= \bigoplus_{\lambda \in \Pc_{r,n}} \bigg( \bigoplus_{{T\in NST(\lambda)}}  \DD F_T \bigg).$ 
 \item
 By Corollary 3.6, $ \DD F_{T_j} \cong  \DD F_{T_j}$ if $T_i,T_j \in NST(\lambda)$ for some $\lambda \in \Pc_{r,n}$ and so we have $f^\la$ isomorphic copies of $ \DD F_{\la}$ in the direct sum (3.3).
 \end{enumerate}
\end{proof}
 Using Proposition 3.1 and Proposition 3.2 we get the next theorem.
\begin{theorem}
\label{thm:one}
\begin{itemize}
\item[(i)] $N_T:=\Dc_Y F_T$ is an irreducible $D_Y$-submodule of $\pi_+( \Oc_X)$.
\item[(ii)] There is a direct sum decomposition
\begin{equation}
\label{eq:direct sum} 
\pi_+ (\Oc_X)=\bigoplus_{\lambda \in \Pc_{r,n} }\bigoplus_{T\in NST(\lambda)}N_T
\end{equation}
\end{itemize}
\end{theorem}

We get in Theorm 3.8 a decomposition of the $\pi_+(\Oc_X)$ into irreducible $\Dc_Y$ modules generated by the higher Specht polynomials.

%
%
%
%

\subsection{Using correspondence between $G$-representations and $D$-modules}

Recall that if $M$ is a semi-simple module over a ring $R$, and $N$ is simple $R$-module, then the isotopic component of $M$ associated with $N$ is the sum $\sum N' \subset M$ of all $N'\subset M$ such that $N' \cong N.$  

\begin{proposition}
For $i=1,\ldots,r$, let $V({\lambda^i})$ be the Specht module corresponding to the partition $\lambda^i \vdash n_i$, $T^i$ a $\lambda^i$-standard tableau and  $M_i:=\Oc_{X_i}$ and $M_i^{\lambda^i}$ the isotopic component of $M_i$ (as $\Oc_{Y_i}$-module) associated with $V({\lambda^i})$. Then 
\begin{enumerate}
\item[(i)]  $\eb_{T^i}(V({\lambda^i}))= \{ \eb_{T^i}(m) | m \in V({\lambda^i} )\}$  is a one dimensional $\bC$-vector space.
\item[(ii)]  $M_i^{\lambda^i}$ is $\Dc_{Y_i} $-module.
\item[(iii)] $\eb_{T^i} (M_i^{\lambda^i})$ si a $\Dc_{Y_i}$-module
\end{enumerate}
\end{proposition}

\begin{proof}
\begin{enumerate}
\item[(i)] In fact we have 
\begin{eqnarray*}
\eb_{T^i} V({\lambda^i}) \cong \eb_{T^i} \bC[\Sc_{n_i}] \eb_{T^i} \cong  \bC \eb_{T^i} \  \mbox{by \ \cite[Theorem 3.9]{Boerner}}
\end{eqnarray*}
\item[(ii)] We only have to prove that $\Dc_{Y_i}  M_i^{\lambda^i} \subset M_i^{\lambda^i}.$ Let $ D \in \Dc_{Y_i}$ and $N$  be a $\bC[\Sc_{n_i}]$-module isomorphic to $V({\lambda^i})$, since by \cite[Corollary 3.5]{DOS} $D$ commute with the elements of the group algebra $\bC[\Sc_{n_i}]$, $D$ is an $\bC[\Sc_{n_i}]$-homomorphism from $N$ into $D(N)$. Then by virtue of the Schur lemma $D(N)=0$ or $D(N)\cong N$ as a $\bC[\Sc_{n_i}]$-module, and $D(N) \subset M_i^{\lambda^i}$. Hence $\Dc_{Y_i}  M_i^{\lambda_i} \subset M_i^{\lambda^i}.$
\item[(iii)] Let  $D \in \Dc_{Y_i},$ we have $ D(\eb_{T^i}(M_i^{\lambda_i})) = \eb_{T^i} (D(M_i^{\lambda^i})) \subset \eb_{T^i} (M_i^{\lambda^i})$, so that  $\eb_{T^i} (M_i^{\lambda^i})$ is a $\Dc_{Y_i}$-module.
\end{enumerate}
\end{proof}

Let us recall  {the correspondence between $G$-representations and $D$-modules} \cite[Paragraph 2.4]{IBK}.
Let $L$ and $K$ be  two extensions fields a field $k$, denote by $T_{K/k}$ the $k$-linear derivations of $K$. We  say that a $T_{K/k}$-module $M$
is  {\it $L$-trivial } if $L\otimes_KM \cong L ^n $ as
$T_{L/k}$-modules. Denote by  $\Mod^{L}(T_{K/k})$ the full subcategory of finitely generated $T_{K/k}$-modules that are $L$-trivial.
 It is immediate that it is closed under taking submodules and quotient modules. Using a lifting $\phi$ , $L$ may be thought of as a $T_{K/k}$-module. 
If $G$ is a finite group let $ \Mod (k[G]) $ be the category of finite-dimensional representations of $k[G]$.
Let now $k\to K \to L$ be a tower of fields such that $K= L^G$.  
Note that the action of $T_{K/k}$ commutes with the action of $G$. If $V$ is a $k[G]$-module,  $L\otimes_kV$ is a $T_{K/k}$-module by $D(l\otimes v)=D(l)\otimes v,\quad D\in T_{K/k}$, and  $(L\otimes_kV)^G$ is a $T_{K/k}$-submodule.

 \begin{proposition}
 \label{prop:equiv}
   The functor
   \begin{displaymath}
\nabla:     \Mod (k[G]) \to \Mod (T_{K/k}), \quad V\mapsto (L\otimes_k V)^G
   \end{displaymath}
   is fully faithful, and defines an equivalence of
   categories $$ \Mod (k[G]) \to \Mod^{L}(T_{K/k}).$$ The quasi-inverse
   of $\nabla$ is the functor 
   \begin{displaymath}
     Loc : \Mod^{L}(T_{K/k}) \to \Mod (k[G]),\quad Loc (M) = (L\otimes_K M)^{\phi(T_{K/k})}.
   \end{displaymath}
\end{proposition} 
\begin{proof}
see \cite[Proposition 2.4]{IBK}
\end{proof}

In the following proposition we take $G=\Sc_{n_i}$, $ K $ the field of fractions of $\Oc_{Y_i}$ and $L$ the field of fractions of $\Oc_{X_i}$ so that $K=L^{\Sc_{n_i}}.$ It is clear that $L$ is a Galois extension of $K$ with Galois $\Sc_{n_i}$, $i=1,\ldots,r$.

\begin{proposition}
For $i=1,\ldots, r$, let $T^i$ be a $\lambda^i$-standard tableau where $\lambda^i \vdash n_i$,  $M_{T^i}:=\eb_{T^i} \Oc_{X_i}$, $V(\lambda^i):=V^{S_0^i} (\lambda^i)$ and $ST(n_i) = \bigcup_{\lambda^i \vdash n_i} ST(\lambda^i).$ Then we have:
\begin{enumerate}
\item $M_{T^i}= \nabla(V({\lambda^i}))$,
\item $M_{T^i}= \eb_{T^i} (M_i^{\lambda^i})$ is simple $\Dc_{Y_i}$-module;
\item $M_i^{\lambda^i} =\displaystyle \bigoplus_{T^i \in ST(n_i)} \eb_{T^i}(M_i^{\lambda^i}).$
\end{enumerate}
\end{proposition}
\begin{proof}
\begin{enumerate}

\item Let us consider  the right $\bC[\Sc_{n_i}]$-module $V=\eb_{T^i} \bC[\Sc_{n_i}]$ where $T^i$ is a $\lambda^i$-standard tableau. This is the image of $\bC[\Sc_{n_i}]$ by right multiplication map $\eb_{T^i}: \bC[\Sc_{n_i}] \to \bC[\Sc_{n_i}]$.  By \cite[Example 2.5]{IBK}, we may turn this map into a left multiplication $ \bC[\Sc_{n_i}]^r \to \bC[\Sc_{n_i}]^r$ and get an image which is isomorphic to $V({\lambda^i})$. Then we have an induced map
$$ \nabla (\bC[\Sc_{n_i}]^r ) \to \nabla (V({\lambda^i})) \subset \nabla (\bC[\Sc_{n_i}]^r ),$$ which is a multiplication by $\eb_{T^i}$ according to \cite[Example 2.5]{IBK}. Then $\nabla(V({\lambda^i}))$ is egal to $\eb_{T^i} \Oc_{X_i} = M_{T^i}.$
\item Since $V({\lambda^i})$ is a simple $\bC[\Sc_{n_i}]$-module, $\nabla (V({\lambda^i}))$ is also a simple $\Dc_{Y_i}$-module.
\item follows from the fact that $1=\sum_{T \in ST(n_i)} \eb_{T}$ and $\eb_{T} (M^{\lambda^i})=0 $ if $T$ is not a $\lambda^i$-tableau.
\end{enumerate}
\end{proof}

\begin{proposition}
For $i=1,\ldots, r$, let $T^i$ be a $\lambda^i$-standard tableau where $\lambda^i \vdash n_i$,  let $M_{T^i}:=\eb_{T^i} \Oc_{X_i}$.
Then 
\begin{enumerate}
 \item $ M_{T^i}=\displaystyle \bigoplus_{ S^i\in ST(\lambda^ii)}\Oc_{Y_i}  F_{T^i}^{S^i}$ as $\Dc_{Y_i}$-module, 
 \item $\Oc_{X_i} = \displaystyle \bigoplus_{\lambda^i \vdash n_i}  \bigg( \bigoplus_{ S^i, T^i\in ST(\lambda^i)}\Oc_{Y_i}  F_{T^i}^{S^i} \bigg) $ as a $\Dc_{Y_i}$-module.
 \end{enumerate}
 \end{proposition}
 
\begin{proof}
\begin{enumerate}
\item For a fixed $S^i \in ST(\lambda^i)$, we know that the polynomial $F_{T^i}^{S^i}$ generate a cyclic  $\bC[\Sc_{n_i}]$-module inside $\Oc_{X_i}$ which is isomorphic to $V({\lambda^i}).$ Then $F_{T^i}^{S^i}\in M_i^{\lambda^i}$ and $M_i^{\lambda^i}= \displaystyle \bigoplus_{S^i,T^i \in ST(\lambda^i)} \bC[\Sc_{n_i}] F_{T^i}^{S^i} \Oc_{Y_i}$ by \cite{Tera-Yama}. Moreover $\eb_{T^i}( F_{T^i}^{S^i})= c F_{T^i}^{S^i}, c\in \bC$ and by Lemma 3.9  $\eb_{T^i} (\bC[\Sc_{n_i}] F_{T^i}^{S^i})=\bC F_{T^i}^{S^i}.$ Hence $M_{T^i}=\eb_{T^i}(M_i^{\lambda^i}) =\displaystyle \bigoplus_{S^i \in ST(\lambda^i)} \Oc_{Y_i}  F_{T^i}^{S^i}.$
\item follows from Proposition 3.11 and  \cite[Theorem 3.6]{DOS}.
\end{enumerate}
\end{proof}

\begin{theorem}
Let  $\lambda \in \Pc_{r,n}$ be an $r$-diagram of size $n$, $T \in NST(\lambda)$ and $M_{T}:=\eb_{T} \OO$. Then
\begin{enumerate}
 \item $ M_{T}=\displaystyle \bigoplus_{ S\in ST(\lambda)} \Oc_{V}  F_{T}^{S}$ as $\DD$-module,
 \item $\OO = \displaystyle \bigoplus_{\lambda \in \Pc_{r,n}}  \bigg( \bigoplus_{ S \in STab (\lambda) T\in NSTab(\lambda)}\Oc_{V}  F_{T}^{S} \bigg) $ as a $\DD$-module.
 \end{enumerate}
\end{theorem}
\begin{proof}
\begin{enumerate}
\item
Suppose  that $\lambda= (\lambda^1, \ldots, \lambda^r),  T=(T^1,\ldots, T^r)$, with $\lambda_i \vdash n_i, T^i \in ST(\lambda^i), i=1,\ldots, r$ and $ \sum n_i=n$
We have that
\begin{eqnarray*}
M_{T} &=& \eb_{T} \OO \\
&=& (\eb_{T^1} \times \ldots \times \eb_{T^r}) (\Oc_{X_1} \otimes \cdots \otimes \Oc_{X_r}) \\
&=& \eb_{T^1} \Oc_{X_1} \otimes \cdots \otimes  \eb_{T^r} \Oc_{X_r} \\
&=& M_{T^1} \otimes \cdots \otimes  M_{T^r} \\
&=&\displaystyle (\bigoplus_{ S^1\in ST(\lambda^1)} \Oc_{Y_1}  F_{T^1}^{S^1}) \otimes \cdots \otimes ( \bigoplus_{ S^2\in ST(\lambda^r)} \Oc_{Y_r}  F_{T^r}^{S^r} ) \  \mbox{by\ Proposition 3.12}\\
&=& \displaystyle \bigoplus_{ S^i\in ST(\lambda^i)} \bigg( \Oc_{Y_1}   \otimes \cdots \otimes \Oc_{Y_r} \bigg) \bigg(F_{T^1}^{S^1} \otimes \cdots \otimes  F_{T^r}^{S^r} \bigg)\  \\
&=& \displaystyle \bigoplus_{ S \in ST(\lambda)} \Oc_{V}  F_{T}^{S} \   \mbox{by\ Lemma 3.4}\  \mbox{with}\  S= (S^1,\ldots ,S^r).
\end{eqnarray*}
\item follows from the fact that $\OO =\Oc_{X_1} \otimes \cdots \otimes \Oc_{X_r}$ and  \\
$\Oc_{X_i}= \displaystyle \bigoplus_{\lambda^i \vdash n_i}  \bigg( \bigoplus_{ S^i, T^i\in ST(\lambda^i)} \Oc_{Y_i}  F_{T^i}^{S^i} \bigg)$ 
\end{enumerate}
\end{proof}

\subsection{ Invariant differential operators and higher Specht polynomials for the symmetric group}

In this subsection we investigate the action of invariant differential operators on higher Specht polynomials.
Let $\lambda \vdash n$,  $T$ a $\lambda$-tableau, and let  $C(T)$  be column stabilizer of $T,$ by  \cite[Lemma 4.4]{DOS}  we know that  for every derivation  ${\bf D}$   such that ${\bf D}(F_T) \neq 0$ then there exists a polynomial  $G$ in  $\mathbb C[x_1,\ldots,x_n]^{C(T)}$, the polynomial ring invariant under the subgroup $C(T)$, such that ${\bf D}(F_T) =F_T G$, we will show that this is also true for the higher Specht polynomials.\\

For $i=1,\ldots, r$, let $\lambda^i \vdash n_i$,  $S^i\in  ST(\lambda^i)$ and $T^i \in ST(\lambda^i)$,  we have that 
$ \mbox{for\ all}\  \sigma \in C(T^i), \sigma(F_{T^i})= \sgn(\sigma) F_{T^i}$ and $\sigma(F^{S^i}_{T^i})= \sgn(\sigma) F^{S^i}_{T^i}.$
\begin{lemma}
For $i=1,\ldots, r$, let $\lambda^i \vdash n_i,$ $T^i, S^i \in  ST(\lambda^i).$
 Then there exists a polynomial $\displaystyle G \in \Oc_{X_i}^{C(T^i)}$ such that $ F_{T^i}^{S^i} =F_{T^i} G.$
\end{lemma}
\begin{proof}
Let  us consider the linear application $\varphi : V(\lambda^i) \to V^{S^i}(\lambda^i)$ defined by $\varphi (F_{T^i})= F_{T^i}^{S^i}$. For every $\sigma \in \bC[\Sc_{n_i}],$ we have that $\varphi(\sigma F_{T^i})= \sigma \varphi (F_{T^i})$, so that  $\varphi$ is a $\bC[\Sc_{n_i}]$-homomorphism  and by the Schur' lemma  $\varphi$ is a $\bC[\Sc_{n_i}]$-isomorphism. Suppose that $x_k$ and $x_l$ occur in the same column of $T^i$, and let $\pi= (k,l)$ the transposition of $k$ and $l$. Then
$$\pi F_{T^i}^{S^i}=\pi \varphi (F_{T^i}) = \varphi(\pi F_{T^i})= \varphi(-F_{T^i})= -F_{T^i}^{S^i}.$$ 
This implies that $(x_k-x_l)$ is a factor of $F_{T^i}^{S^i}.$ This holds for each linear factor of $F_{T^i}$, so that $F_{T^i}$ divides $F_{T^i}^{S^i}$. Hence there exists a polynomial  $\displaystyle G \in \mathbb C[x_{n_1+\cdots +n_{i-1}+1},\ldots,x_{n_1 +\cdots +n_i}] $ such that $F_{T^i}^{S^i}= F_{T^i} G $. 
Let now $ \sigma \in C(T^i)$, we get $$ \sigma G=\sigma \bigg( \frac{F_{T^i}^{S^i}}{F_{T^i}}\bigg)  =\frac{\sigma F_{T^i}^{S^i}}{\sigma F_{T^i}}= \frac{\sgn \sigma F_{T^i}^{S^i}}{\sgn \sigma F_{T^i}}= G.$$
Then $ \displaystyle G \in \mathbb C[x_{n_1+\cdots +n_{i-1}+1},\ldots,x_{n_1 +\cdots +n_i}]^{C(T^i)}.$
\end{proof}

\begin{lemma}
For $i=1,\ldots, r$, let $\lambda^i \vdash n_i,$  $ T^i, S^i \in  ST(\lambda^i),$
and ${\bf D}$ a derivation in $\Dc_{Y_i}$ such that ${\bf D}(F_{T^i}^{S^i}) \neq 0$. Then there exists a polynomial  $G \in \Oc_{X_i}^{C(T)}$ such that ${\bf D}(F_{T^i}^{S^i}) =F_{T^i}G.$
\end{lemma}
\begin{proof}

Let ${\bf D}$ a derivation in $ \Dc_{Y_i}$, we have that
\begin{eqnarray*}
{\bf D} (F_{T^i}^{S^i}) 
&=&  {\bf D}(F_{T^i} G') \  \mbox{where}\ G' \mbox{is\ a\ polynomial\ in } \Oc_{X_i}^{C(T^i)} \mbox{by\ Lemma 3.14}  \\ 
&=&  {\bf D}(F_T) G' + F_T {\bf D}(G') \\
&=&F_T G''G' + F_T {\bf D}(G') \  \mbox{where}\ G''\in \Oc_{X_i}^{C(T^i)} \  \mbox{by \cite[Lemma 4.4]{DOS}}\\
&=&F_T (G'' G' +  {\bf D}(G') ) \  \mbox{with}\ G', G''  \in \Oc_{X_i}^{C(T^i)}.
\end{eqnarray*}
Now let $ \pi \in {C(T^i)}$ we have 
\begin{eqnarray*}
\pi (G''G' +  {\bf D}(G'))&=& \pi (G'')\pi(G') + \pi {\bf D}(G') \\
&=& G''G' +  {\bf D}(\pi G')\ \mbox{since}\ G', G''  \in \Oc_{X_i}^{C(T^i)} \\
 &=& G''G' +  {\bf D}(G')\ 
 \end{eqnarray*} 
 Then   $G''G' +  {\bf D}(G') \in \Oc_{X_i}^{C(T^i)}.$ Set $G= G''G' +  {\bf D}(G')$ and we get  ${\bf D}(F_{T^i}^{S^i}) =F_{T^i}G.$
\end{proof}

\begin{proposition}
let $\lambda \in \Pc_{r,n}$ be an $r$-diagram, $T, S \in ST(\lambda)$ and ${\bf D}$ a derivation in $\DD$ such that ${\bf D}(F_{T}^{S}) \neq 0$. Then there exists a polynomial  $G \in \OO^{C(T)}$, where $C(T)= C(T^1) \times \cdots \times  C(T^r)$ such that ${\bf D}(F_{T}^{S}) =F_{T}G.$
\end{proposition}
\begin{proof}
For $i=1\ldots r$, there exists a derivation ${D}_i \in \Dc_{Y_i}$ with $D_i(F_{T^i}^{S^i}) \neq 0,$ such that $ {\bf D}= D_1 \otimes \cdots \otimes D_r.$ Then

\begin{eqnarray*}
{\bf D} (F_T^S)  &=& (D_1 \otimes \cdots \otimes D_r) (F_T^S) \\
&=& (D_1 \otimes \cdots \otimes D_r) (F_{T^1}^{S^1} \otimes \cdots \otimes F_{T^r}^{S^r}) \mbox{by\ Lemma 3.4}\\
&=& D_1(F_{T^1}^{S^1}) \otimes \cdots \otimes D_r(F_{T^r}^{S^r})  \\
&=& F_{T^1} G_1 \otimes \cdots \otimes F_{T^r}G_r\  \mbox{where}\ G_i \in \Oc_{X_i}^{C(T^i)}\  \mbox{by Lemma 3.15} \\
&=& (F_{T^1} \otimes \cdots \otimes F_{T^r})( G_1 \otimes \cdots \otimes G_r)\  \mbox{with}\ G_1\otimes \cdots \otimes G_r \in \OO^{C(T)} \\
&=& F_T G\  \mbox{where}\  G= G_1 \otimes \cdots \otimes G_r\  \mbox{by\ Lemma 3.4}
\end{eqnarray*}
\end{proof}

\begin{proposition}
For $i=1,\ldots, r$, let $\lambda^i \vdash n_i,$  $ T^i, S^i \in  ST(\lambda^i)$ and ${\bf D} \in \Dc_{Y_i}$ such that ${\bf D}(F_{T^i}^{S^i}) \neq 0$ for $S^i, T^i \in ST(\lambda^i)$. Then the image of the $\bC[\Sc_{n_i}]$-module $V^{S^i}({\lambda^i})$  by ${\bf D}$ is an $\bC[\Sc_{n_i}]$-module  isomorphic to $V^{S^i}({\lambda^i})$. 
\end{proposition}

\begin{proof}
Let  $\lambda^i \vdash n_i, {\bf D} \in \Dc_{Y_i}$ such that ${\bf D}(F_{T^i}^{S^i}) \neq 0$ for $S^i,T^i \in ST(\lambda^i)$ and set $W^{S^i}_{\bf D}({\lambda^i}):= {\bf D} (V^{S^i}({\lambda^i}))$ the image of the  module $V^{S^i}({\lambda^i})$  under the map ${\bf D}$. Since the $\bC$-vector space $V^{S^i}({\lambda^i})$ is  equipped with a basis $\C^{S^i}(\lambda^i) = \{ F_{T^i}^{S^i}; T^i\in ST (\lambda^i) \}$,  $W^{S^i}_{\bf D}({\lambda^i})$ is the vector space  spanned by the set $\{ {\bf D}(F_{T^i}^{S^i});\  T^i\in ST(\lambda^i) \}$. The elements of $\{ F_{T^i}^{S^i} ; T^i\in ST(\lambda^i) \}$ are linearly independent over $\Dc_{Y_i}$, otherwise the direct sums in  Proposition 3.12 cannot hold. It follows that the elements  in $\{ {\bf D}(F_{T^i}^{S^i});\ T^i \in ST(\lambda^i) \}$ are linear independent over $\bC$.  Hence $\{ {\bf D}(F_{T^i}^{S^i}) ;\  T^i \in ST(\lambda^i) \}$ is a basis of $W^{S^i}_{\bf D}({\lambda^i})$ over $\bC$. Since ${\bf D}$ commute with  elements of $\bC[\Sc_{n_i}],\ W^{S^i}_{\bf D}({\lambda^i})$ is an $\bC[\Sc_{n_i}]$-module  isomorphic to $V^{S^i}({\lambda^i}).$
\end{proof}

\begin{theorem}
Let  $\lambda \in \Pc_{r,n}$ be an $r$-diagram of size $n$, $T \in NST(\lambda)$  and ${\bf D} \in \DD$ such that ${\bf D}(F_{T}^{S}) \neq 0$ for $S \in ST(\lambda)$. Then the image of the $\bC[\Sc_{n_1} \times \cdots \times \Sc_{n_r}]$-module $V^{S}({\lambda})$  by ${\bf D}$ is a $\bC[\Sc_{n_1} \times \cdots \times \Sc_{n_r}]$-module  isomorphic to $V^{S}({\lambda})$. In others words, the action of the differential operators of $\DD$ on the higher Specht polynomials generate  isomorphic copies of the corresponding  module.
\end{theorem}
\begin{proof}
Let  $\lambda$ be an $r$-diagram of size $n$, $T \in NST(\lambda)$  and ${\bf D} \in \DD$ such that ${\bf D}(F_{T}^{S}) \neq 0$ for $S \in ST(\lambda)$. Then ${\bf D}$ may be written as ${\bf D}= { D_1}\otimes \cdots \otimes  { D_r}$ where ${ D_i} \in \Dc_{Y_i},
\ i=1,\ldots, r $.
 $$
  {\bf D} (F_{T}^{S}) = ({ D_1}\otimes \cdots \otimes  { D_r}) (F_{T^1}^{S^1} \otimes \cdots \otimes  F_{T^r}^{S^r}) 
   =  D_1(F_{T^1}^{S^1}) \otimes \cdots \otimes  D_r(F_{T^r}^{S^r})  \neq 0,$$
 so that  ${ D_i} F_{T^i}^{S^i} \neq 0,\ i=1\ldots, r.$ Then by Proposition 3.17, ${ D_i} F_{T^i}^{S^i}$  generate a $\bC[\Sc_{n_i}]$-module isomorphic to $V^{S^i} (\lambda^i),\ i=1,\ldots, r$. Hence $D_1(F_{T^1}^{S^1}) \otimes \cdots \otimes  D_r(F_{T^r}^{S^r}) $ generate a $\bC[\Sc_{n_1} \times \cdots \times \Sc_{n_r}]$-module isomorphic to $V^{S^1} (\lambda^1) \otimes \cdots \otimes  V^{S^r} (\lambda^r) \cong  V^{S}({\lambda})$ by LEmma 3.4.

\end{proof}

%

\end{document}